\documentclass[a4paper,12pt]{article}
\usepackage{amsfonts}
\usepackage{amstext}
\usepackage{amsthm}
\usepackage{amsmath}
\usepackage{amssymb}
\usepackage{latexsym}
\usepackage{graphicx}
\usepackage[latin1]{inputenc}
\newcommand{\R}{\Bbb{R}}

\newcommand{\s}{\Bbb{S}}

\newtheorem{teor}{Theorem}[section]
\newtheorem{propo}{Proposition}[section]

\newtheorem{lema}{Lemma}[section]
\newtheorem{cor}{Corollary}[section]

\newcommand{\n}{\noindent}
\newcommand{\vs}{\vspace}

\begin{document}

\title{On spectral optimization for $p$-Laplace type operators
}

\author{\textbf{Raul Fernandes Horta \footnote{\textit{E-mail addresses}: raul.fernandes.horta@gmail.com
 (R. F. Horta)}}\\ {\small\it Departamento de Matem\'{a}tica,
Universidade Federal de Minas Gerais,}\\ {\small\it Caixa Postal
702, 30123-970, Belo Horizonte, MG, Brazil}\\
\textbf{Marcos Montenegro \footnote{\textit{E-mail addresses}:
montene@mat.ufmg.br (M. Montenegro)}}\\ {\small\it Departamento de Matem\'{a}tica,
Universidade Federal de Minas Gerais,}\\ {\small\it Caixa Postal
702, 30123-970, Belo Horizonte, MG, Brazil}}

\date{}

\maketitle

\markboth{abstract}{abstract}
\addcontentsline{toc}{chapter}{abstract}

\hrule \vspace{0,2cm}

\n {\bf Abstract}

Sharp uniform estimates, both lower and upper, are obtained for fundamental frequencies of $p$-Laplace type operators generated by quadratic forms. Optimal constants are computed, rigidity of the upper estimate is derived, anisotropic attainability of the lower constant is established, along with the characterization of corresponding anisotropic extremizers for certain membranes. Sharp quantitative inequalities associated with lower constants are also proved, thereby providing anisotropic stability. The optimal lower constant remains positive even when the condition of uniform strict ellipticity of the operators is removed. Our sharp lower estimate can be viewed as an isoanisotropic counterpart of the two-dimensional Faber-Krahn isoperimetric inequality.

\vspace{0.5cm}
\hrule\vspace{0.2cm}

\section{Introduction}

A field of research that has aroused great interest in the mathematical community for several decades is the study of problems linking shapes of domains and elliptic operators to the corresponding spectra. Most of their solutions demand techniques in areas as calculus of variations, convex analysis, elliptic PDEs, geometry and measure geometric theory. A well-known prototype of the first group of problems is the celebrated Faber-Krahn isoperimetric inequality which states that the first Dirichlet eigenvalue $\lambda_1(\Omega)$ of the Laplace operator in a bounded domain $\Omega \subset \R^n$, also called the fundamental frequency of $\Omega$, satisfies

\begin{equation} \label{FK}
\lambda_1(\Omega) \geq \lambda_1(B)
\end{equation}
for every domain $\Omega$ with $\vert \Omega \vert = \vert B \vert$, where $B$ denotes the unit $n$-Euclidean ball and $\vert \cdot \vert$ stands for the Lesbegue measure of a measurable subset of $\R^n$. Moreover, equality holds in \eqref{FK} if, and only if, $\Omega$ is equal to $B$, up to a translation and a set of zero capacity. Inequality \eqref{FK}, as the name itself suggests, was independently proved by Faber \cite{F} and Krahn \cite{Kr} for any dimension $n \geq 2$ through the Schwarz symmetrization technique. A number of variants of it has been established along of the history. For example, related inequalities involving the first nonzero eigenvalue of the Laplace operator under different boundary conditions are: Bossel-Daners inequality \cite{B, BD, D, DK} (Robin condition) (see also \cite{DPG, DPP}), Szegö-Weinberger inequality \cite{S, W} (Neumann condition) (see also \cite{BH}) and Brock-Weinstock inequality \cite{Br, We} (Steklov condition), being the first of them of lower kind, while the other two ones are upper. Different quantitative forms and stability of geometric inequalities have also been widely discussed for the spectrum of elliptic operators; among an extensive bibliography we refer to \cite{BGI, BDPV, BFNT, FZ} for some quantitative inequalities and \cite{FMP, HN, M} concerning stability. Besides extensions of \eqref{FK} to more general operators have been of interest, see \cite{A, BDVV, BCV, CMT} for some of these developments. For an overview on a variety of problems, improvements and open questions on several related topics, we quote the excellent monographes: the collection of contributions \cite{H2} and the books \cite{BB, H1, HP}.

A particularly important counterpart of \eqref{FK} related to the present work was established by Belloni, Ferone and Kawohl \cite{BFK} in an anisotropic context. Precisely, let $H: \R^n \rightarrow \R$ be a convex function of $C^1$ class in $\R^n \setminus \{0\}$ satisfying the homogeneity and strict positivity assumptions:

\begin{equation}
H(t\xi) = \vert t \vert H(\xi),\ \forall t \in \R,\ \xi \in \R^n \tag{H}
\end{equation}
and

\begin{equation}
\kappa_1 \vert \xi \vert \leq H(\xi) \leq \kappa_2 \vert \xi \vert,\ \forall \xi \in \R^n, \tag{P}
\end{equation}
where $\kappa_1 = \kappa_1(H)$ and $\kappa_2 = \kappa_2(H)$ are positive constants depending of $H$.

Given any number $p > 1$, consider the anisotropic differential operator in divergence form depending of $p$ and $H$:

\[
\Delta^H_p u := -{\rm div} (H^{p-1}(\nabla u) \nabla H(\nabla u)).
\]
It arises naturally from derivation of the energy functional on $W^{1,p}_0(\Omega)$ defined by

\[
{\cal E}_{p,H}(u) = \int_\Omega H^p(\nabla u)\, dx,
\]
since $(p,H)$-harmonic functions ({\it i.e.} weak solutions of $\Delta^H_p u = 0$) are critical points of ${\cal E}_{p,H}$.

The first Dirichlet eigenvalue $\lambda^H_{1,p}(\Omega)$ of the operator $\Delta^H_p$, also called the $H$-anisotropic fundamental $p$-frequency of $\Omega$, is variationally characterized as

\[
\lambda^H_{1,p}(\Omega) = \inf\left\{ {\cal E}_{p,H}(u):\ u \in W^{1,p}_0(\Omega),\ \Vert u \Vert_{L^p(\Omega)} = 1\right\}.
\]
From the conditions (H) and (P), it follows that $\lambda^H_{1,p}(\Omega)$ is positive and $\Delta^H_p$ is uniformly elliptic. The latter fact along with (H) and $C^2$ regularity of $H$ outside the origin imply that $\lambda^H_{1,p}(\Omega)$ is the unique principal eigenvalue of $\Delta^H_p$, which is also simple. In addition, all eigenfunctions of $\Delta^H_p$ in $W^{1,p}_0(\Omega)$ belong to $C^{1,\beta}(\Omega)$ and also extends $C^{1,\beta}$ at points on $C^{1,\alpha}$ parts of the boundary of $\Omega$ where $\alpha > \beta$. For the part of regularity, we refer to \cite{To, To1} and references therein.

Consider the unit anisotropic ball $B_H = \left\{ \xi \in \R^n:\ H(\xi) \leq 1\right\}$ and its polar body defined by

\[
B_H^\circ = \{\xi \in \R^n:\, \langle \xi, \eta\rangle \leq 1,\ \eta \in B_H\}.
\]
The anisotropic Faber-Krahn inequality established in \cite{BFK} states that, for any bounded domain $\Omega$ with $\vert \Omega \vert = \vert B_H^\circ \vert$,

\begin{equation} \label{AFK}
\lambda^H_{1,p}(\Omega) \geq \lambda^H_{1,p}(B_H^\circ)
\end{equation}
and, in addition, equality holds in \eqref{AFK} if, and only if, $\Omega$ is equal to $B_H^\circ$, module a translation and a set of zero capacity.

It also deserves to be remarked that the above inequality provides a sharp lower uniform estimate of $\lambda^H_{1,p}(\Omega)$ in terms of the measure $\vert \Omega \vert$, namely

\[
\lambda^H_{1,p}(\Omega) \geq \vert B_H^\circ \vert^{\frac pn} \lambda^H_{1,p}(B_H^\circ) \, \vert \Omega \vert^{-\frac pn}.
\]

As with most eigenvalue optimization problems where domains vary, optimization problems that involve the dependence on second order elliptic operators in a fixed domain are usually quite interesting. Within this spirit, considerable advances have been achieved in optimization problems that consist in finding, given a fixed bounded domain $\Omega$, uniform estimates of eigenvalues with respect to the coefficients of elliptic operators and/or other involved parameters, with optimal lower and upper constants being explicitly computed in some cases. For example, this question has been addressed to Dirichlet eigenvalues $\lambda_k(V)$ of the Schrödinger operator

\[
{\cal L}_S u = -\hbar \Delta u + V(x) u\ \ {\rm in}\ \Omega,
\]
where $\hbar$ denotes the Planck's constant and $V$ is a potential function satisfying the conditions

\[
\Vert V \Vert_{L^p(\Omega)} \leq \kappa\ \ {\rm and}\ \ \kappa_1 \leq V \leq \kappa_2\ \ {\rm in}\ \Omega
\]
for a priori fixed constants $\kappa$, $\kappa_1$ and $\kappa_2$, we refer to \cite{CMZ, EK, E, K, KS, N, T, ZW} for the one-dimensional case and \cite{AM, BBV, BGRV, CGIKO, Eg, H, MRR} and Chapter 9 of \cite{H2} for higher dimensions.

Sharp estimates have been also obtained for Dirichlet eigenvalues $\lambda_k(\sigma)$ of the elliptic operator

\[
 {\cal L}_C u = -{\rm div}(\sigma(x) \nabla u)\ \ {\rm in}\ \Omega
\]
for conductivity functions $\sigma$ normalized simultaneously by the uniform and $L^p$ constraints above with constants $\kappa_1, \kappa_2 > 0$, see \cite{Ba, CL, EKo, Tr} for some developments on this subject.

Inequalities as \eqref{FK} and \eqref{AFK} are of geometric nature once one treats sharp lower estimate of the first Dirichlet eigenvalue of elliptic operators on bounded domains in terms of their Lebesgue measures. A natural query is whether there are anisotropic counterparts of \eqref{FK} and \eqref{AFK}. A positive answer to this question depends on finding some prescribed ``measure" condition for functions $H$. Given that an appropriate normalization has been defined, we are faced with two fundamental questions on a fixed bounded domain $\Omega$. More precisely, we have:

\begin{itemize}
\item[(A)] Are there explicit optimal constants $\lambda^{\min}_{1,p}(\Omega)$ and $\lambda^{\max}_{1,p}(\Omega)$ such that

\[
\lambda^{\min}_{1,p}(\Omega) \leq \lambda^H_{1,p}(\Omega) \leq \lambda^{\max}_{1,p}(\Omega)
\]
for every convex function $H$ of ``unit measure" satisfying (H) and (P)?

\item[(B)] Do extremal functions $H$ exist for both inequalities in (A)? Is there any characterization of them?
\end{itemize}
The positivity of $\lambda^{\min}_{1,p}(\Omega)$ and finiteness of $\lambda^{\max}_{1,p}(\Omega)$ must necessarily depend on the set of anisotropies $H$ to be considered. Inequalities in (A) exhibit sharp uniform estimates on $H$ for the first Dirichlet eigenvalue of the operator $\Delta_p^H$. Notice also that the first of them is an isoanisotropic parallel of the $n$-dimensional Faber-Krahn isoperimetric inequality. Functions $H$ solving (B) are called anisotropic extremizers.

The purpose of this paper is to discuss (A) and (B) in the plane regarding the class of functions $H$ generated by positive quadratic forms so that $\Delta^H_p$ is strictly and uniformly elliptic, namely functions $H$ under the form

\[
H(x,y) = (\alpha x^2 + 2 \beta xy + \gamma y^2)^{1/2}
\]
satisfying (P) with positive constants $\kappa_1$ and $\kappa_2$ independent of $H$. Our family of elliptic operators includes the usual Laplace and $p$-Laplace operators, since $\Delta = \Delta_2^{H_0}$ and $\Delta_p = \Delta_p^{H_0}$ for the anisotropy  $H_0(x,y) = \vert (x,y) \vert = (x^2 + y^2)^{1/2}$.

The question (A) will be completely solved for the referred class. Comprehensive answers to the question (B) are also provided into three parts: we first show the rigidity of the second inequality in (A), being attained exactly by $H_0$; we prove the existence of at least one anisotropic extremizer for the first inequality; we characterize all anisotropic extremizers for the first inequality in the case of disks and of a family of rectangles, making it clear that the quantity (multiplicity) of the optimal functions depends on the shape of the domain $\Omega$. We also present sharp quantitative anisotropic estimates related to the lower constant $\lambda^{\min}_{1,p}(\Omega)$. Finally, we discuss (A) and (B) when the uniform ellipticity condition is relaxed, in particular we show that the optimal lower constant remains positive, while anisotropic extremizers no longer exist.

In practice, the paper develops a spectral optimization theory for a class of $p$-Laplace type operators in $\R^2$ in which most of the results are sharp and optimal constants and extremizers are presented explicitly. The leading arguments are based on maximum principles, on monotonicity property of eigenvalues and mainly on a key relation stated in Theorem \ref{T6} which is the bridge that allows to link the quadratic anisotropic environment to the study of anisotropic stability of optimal lower constants by means of aforementioned quantitative estimates.

The remainder of paper is organized into four sections. In Section 2 we introduce basic notations and the entire setup of quadratic anisotropic optimization. In Section 3 we present all statements related to (A) and (B). Section 4 is devoted to proofs of the results dealing with anisotropic maximization, whereas Section 5 is dedicated to all proofs related to anisotropic minimization.

\section{The spectral optimization setting}

We will use from now on the letter $Q$ in place of $H$ in the notations of anisotropic eigenvalues, since the class of functions $H$ to be considered is in correspondence with positive quadratic forms $Q$ by means of the relation $H = Q^{1/2}$.

The set of all positive quadratic forms on $\R^2$ is given by

\[
{\cal Q} := \{Q(x,y) = \alpha x^2 + 2 \beta xy + \gamma y^2:\, \alpha, \gamma > 0,\ \beta \geq 0\ \ {\rm and}\ \ \beta^2 < \alpha \gamma\}.
\]
For each $Q \in {\cal Q}$, set

\[
Q_{\min} = \min_{\vert(x,y)\vert = 1} Q(x,y)\ \ {\rm and}\ \ Q_{\max} = \max_{\vert(x,y)\vert = 1} Q(x,y),
\]
where $\vert \cdot \vert$ denotes the Euclidean norm.

For $a \in (0,1]$, we consider the subset of ${\cal Q}$ of normalized functions as

\[
{\cal Q}^a := \{Q \in {\cal Q}:\, Q_{\min} \geq a\ \ {\rm and}\ \  Q_{\max} = 1\}.
\]
For any $Q \in {\cal Q}^a$, clearly we have

\begin{equation} \label{LU}
a \vert(x,y)\vert^2 \leq Q(x,y) \leq \vert(x,y)\vert^2,\ \forall (x,y) \in \R^2.
\end{equation}
For $a = 0$, we denote by ${\cal Q}^0$ the subset of ${\cal Q}$ with relaxed strict coercivity, that is

\[
{\cal Q}^0 := \{Q \in {\cal Q}:\, Q_{\max} = 1\}.
\]

Let $\Omega \subset \R^2$ be a bounded domain and let $p > 1$ be a fixed parameter. The $Q$-anisotropic fundamental $p$-frequency is defined for $Q \in {\cal Q}$ as

\[
\lambda^Q_{1,p}(\Omega) := \inf \ \left\{ \iint_\Omega Q^{\frac p2}(\nabla u)\, dA:\ u \in W^{1,p}_0(\Omega),\ \Vert u \Vert_p = 1 \right\},
\]
where $dA$ denotes the usual area element.

As mentioned in the introduction, it is a classical result that the infimum is always attained by a positive function $\varphi_p \in C^{1,\beta}(\Omega)$, which is a principal eigenfunction associated to the principal Dirichlet eigenvalue $\lambda^Q_{1,p}(\Omega)$ of the $Q$-anisotropic operator

\[
\Delta^Q_p u := - {\rm div} \left( Q^{\frac{p-1}{2}}(\nabla u) \nabla Q^\frac{1}{2}(\nabla u)\right).
\]
In addition, $\varphi_p$ is $C^{1,\beta}$ on $C^{1,\alpha}$ parts of the boundary of $\Omega$ whenever $\alpha > \beta$.

Notice also that the operator $\Delta^Q_p$ is uniformly elliptic for each $Q \in {\cal Q}^a$ in a uniform sense since the corresponding ellipticity constants depend only on the fixed number $a$. Taking into account this uniformity, we consider two anisotropic optimization problems for each $a \in [0, 1]$:

\[
\lambda_{1,p}^{\min}({\cal Q}^a,\Omega):= \inf_{Q \in {\cal Q}^a} \lambda_{1,p}^Q(\Omega)\ \ {\rm and}\ \ \lambda_{1,p}^{\max}({\cal Q}^a, \Omega):= \sup_{Q \in {\cal Q}^a} \lambda_{1,p}^Q(\Omega).
\]
From the definition of ${\cal Q}^a$ and \eqref{LU}, one easily checks that

\begin{equation} \label{LU1}
a^{p/2} \lambda_{1,p}(\Omega) \leq \lambda_{1,p}^{\min}({\cal Q}^a,\Omega) \leq \lambda_{1,p}^{\max}({\cal Q}^a, \Omega) \leq \lambda_{1,p}(\Omega),
\end{equation}
where $\lambda_{1,p}(\Omega)$ denotes the principal Dirichlet eigenvalue of the $p$-Laplace operator. In particular, both constants are finite, and clearly positive for $a > 0$. Furthermore, when $a = 1$, the definition of ${\cal Q}^a$ and inequalities in \eqref{LU1} ensure that ${\cal Q}^1 = \{\vert \cdot \vert^2\}$ and

\[
\lambda_{1,p}^{\min}({\cal Q}^1,\Omega) = \lambda_{1,p}^{\max}({\cal Q}^1,\Omega) = \lambda_{1,p}(\Omega).
\]

The questions (A) and (B) presented in the introduction can now be rephrased for $a \in [0,1)$ as

\begin{itemize}
\item[(A1)] Can the optimal anisotropic constants $\lambda_{1,p}^{\min}({\cal Q}^a,\Omega)$ and $\lambda_{1,p}^{\max}({\cal Q}^a, \Omega)$ be computed?

\item[(B1)] Do the constants in (A1) admit anisotropic extremizers? Is it possible to characterize them?
\end{itemize}

Finally, we introduce the set of anisotropies with non-normalized coercivity for each $a \in (0,1)$:

\begin{eqnarray*}
{\cal Q}_{nn}^a &:=& \{Q \in {\cal Q}:\, Q_{\min} \geq a Q_{\max} \} \\
&=& \{Q \in {\cal Q}:\, Q(x,y) \geq a Q_{\max}\, \vert (x,y) \vert^2,\ \forall (x,y) \in \R^2\}.
\end{eqnarray*}
Since $Q^\varepsilon(x,y) = \varepsilon \vert (x,y) \vert^2$ belongs to ${\cal Q}_{nn}^a$ but not to ${\cal Q}^a$ for any $0 < \varepsilon < a$, we note that there are anisotropic operators $\Delta^Q_p$ associated to $Q \in {\cal Q}_{nn}^a$ with arbitrarily small ellipticity constant.

\section{Main results on spectral optimization}

Our first result answers affirmatively both questions (A1) and (B1) for $\lambda_{1,p}^{\max}({\cal Q}^a, \Omega)$.

\begin{teor}[Rigidity] \label{T1}
For any $a \in [0, 1)$, we have

\[
\lambda_{1,p}^{\max}({\cal Q}^a, \Omega) = \lambda_{1,p}(\Omega).
\]

Moreover, $Q = \vert \cdot \vert^2 \in {\cal Q}^a$ is the unique anisotropic extremizer for $\lambda_{1,p}^{\max}({\cal Q}^a, \Omega)$ provided that $\partial \Omega$ is $C^{1,\alpha}$ by parts.
\end{teor}

Regarding the anisotropic minimization problem we need to introduce a family of elements of ${\cal Q}^a$ and of orthogonal $2 \times 2$ matrices in order to state the next theorems.

For $\alpha \in [a, 1]$, we consider the anisotropic function $Q_\alpha(x,y) = \alpha x^2 + 2 \beta(\alpha) xy + \gamma(\alpha) y^2$ with coefficients

\[
\beta(\alpha) = \sqrt{(1 - \alpha) (\alpha - a)}\ \text{ and }\ \gamma(\alpha) = 1 + a - \alpha.
\]
As will be justified later, $Q_\alpha \in {\cal Q}^a$ and $(Q_\alpha)_{\min} = a$ for every $\alpha \in [a, 1]$ .

Let also ${\cal O}$ be the set of all rotations in the plane of angle between $0$ and $\pi/2$. For $A \in {\cal O}$ we set $\Omega_A = A^T(\Omega)$, that is $\Omega_A$ is the rotation of the domain $\Omega$ by the orthogonal matrix $A^T$.

The second result completes the solution of the question (A1) and establishes the existence of at least one anisotropic extremizer for $\lambda_{1,p}^{\min}({\cal Q}^a, \Omega)$ in the case $a \in (0,1)$.

\begin{teor}[Attainability] \label{T2}
For any $a \in (0, 1)$, we have

\begin{eqnarray*}
\lambda_{1,p}^{\min}({\cal Q}^a, \Omega) &=& \min \left\{ \lambda_{1,p}^{Q_a}(\Omega_A):\ \text{for any } A \in {\cal O} \right\}\\
&=& a^{p/2} \min \left\{ \lambda_{1,p}(\Omega^a_A):\ \text{for any } A \in {\cal O} \right\},
\end{eqnarray*}
where $\Omega^a = \{(x, \sqrt{a}\, y):\ (x,y) \in \Omega\}$.

Moreover, all anisotropic extremizers for $\lambda_{1,p}^{\min}({\cal Q}^a, \Omega)$ belong to $\{Q_\alpha:\ \alpha \in [a, 1]\}$ provided that $\partial \Omega$ is $C^{1,\alpha}$ by parts.

\end{teor}

Theorems \ref{T1} and \ref{T2} applied to the normalized function $Q_{\max}^{-1}\, Q \in {\cal Q}^a$ for $Q \in {\cal Q}_{nn}^a$ readily yield:

\begin{cor} \label{C.0}
For any $a \in (0, 1)$, the sharp non-normalized anisotropic estimates hold

\[
\lambda_{1,p}^{\min}({\cal Q}^a, \Omega)\, Q_{\max}^{p/2} \leq \lambda^Q_{1,p}(\Omega) \leq \lambda_{1,p}^{\max}({\cal Q}^a, \Omega)\, Q_{\max}^{p/2}
\]
for every $Q \in {\cal Q}_{nn}^a$ with explicit optimal constants $\lambda_{1,p}^{\min}({\cal Q}^a, \Omega)$ and $\lambda_{1,p}^{\max}({\cal Q}^a, \Omega)$.
\end{cor}

A second consequence, which will be easily checked, states that

\begin{cor} \label{C.1}
Assume $\partial \Omega$ is $C^{1,\alpha}$ by parts. For any $a \in (0, 1)$, we have

\[
a^{p/2} \lambda_{1,p}(\Omega) < \lambda_{1,p}^{\min}({\cal Q}^a,\Omega) < \lambda_{1,p}^{\max}({\cal Q}^a, \Omega).
\]
\end{cor}
It is natural to wonder at this point what happens if we relax the condition $Q_{\min} \geq a$ with $a \in (0,1]$ to $Q_{\min} > 0$, in other words, if we consider the anisotropic minimization problem in the set ${\cal Q}^0$. It is easily checked that

\[
\lambda_{1,p}^{\min}({\cal Q}^0,\Omega) = \inf_{a \in (0,1]} \lambda_{1,p}^{\min}({\cal Q}^a,\Omega).
\]
The next result ensures that $\lambda_{1,p}^{\min}({\cal Q}^0,\Omega)$ remains positive, however, anisotropic extremizer no longer exist, so the condition of strict coercivity for some number $a \in (0,1]$ is necessary for the existence of such an extremizer.

\begin{teor} \label{T3a}
The optimal lower constant $\lambda_{1,p}^{\min}({\cal Q}^0, \Omega)$ is always positive. However, anisotropic extremizers do not exist in ${\cal Q}^0$ provided that $\partial\Omega$ is $C^{1,\alpha}$ by parts.
\end{teor}

Unlike the upper constant $\lambda_{1,p}^{\max}({\cal Q}^a, \Omega)$, the lower $\lambda_{1,p}^{\min}({\cal Q }^a, \Omega)$ depends on the sets ${\cal Q }^a$. In general, we have $\lambda_{1,p}^{\min}({\cal Q }^a, \Omega) \leq \lambda_{1,p}^{\min}({\cal Q }^b, \Omega)$ whenever $a \leq b$ since ${\cal Q }^b \subset {\cal Q }^a$.

Other important topic concerns the stability of $\lambda_{1,p}^{\min}({\cal Q }^a, \Omega)$ in relation to the sets ${\cal Q }^a$. This follows directly from the sharp quantitative inequality:

\begin{teor}[Quantitative estimate I] \label{T20}
Let $a \in (0,1)$ and $b \in [a,1)$. The upper quantitative inequality states that

\[
\frac{\lambda_{1,p}^{\min}({\cal Q}^{b},\Omega)}{\lambda_{1,p}^{\min}({\cal Q}^{a},\Omega)} - 1 \leq p\sqrt{\frac{b^{p-1}(b-a)}{a^p(1-a)}}.
\]

Moreover, it is sharp in the sense that equality implies $b = a$.
\end{teor}

We also establish a reverse sharp quantitative inequality.

\begin{teor}[Quantitative estimate II] \label{T201}
Let $a \in (0,1)$ and $b \in [a,1)$. The lower quantitative inequality

\[
\lambda_{1,p}^{\min}({\cal Q}^b,\Omega) - \lambda_{1,p}^{\min}({\cal Q}^a,\Omega) \geq C(a,b,p,\Omega)(b-a)
\]
holds for
\begin{equation*}
  C(a,b,p,\Omega) = \begin{cases}
                     \frac{pa^{(p-2)/2}}{2} c_0(p,\Omega), & \mbox{if } \ p \geq 2 \\
                     \frac{pb^{(2-p)/2}}{2} \lambda_{1,p}(\Omega)^{(p-2)/2} c_{0}(p,\Omega)^{2/p} , & \mbox{if} \ 1 < p < 2,
                   \end{cases}
\end{equation*}
where
\[
c_0(p,\Omega):= \inf_{A \in {\cal O}}\ \inf_{u \in W^{1,p}_0(\Omega)} \ \left\{ \iint_{\Omega_A} \vert D_x u \vert^p\, dA:\ \Vert u \Vert_p = 1 \right\}
\]
is a positive constant.

Moreover, both are sharp provided that $\partial \Omega$ is $C^{1,\alpha}$.
\end{teor}

Regardless of the shape of the membrane $\Omega$, Theorem \ref{T2} guarantees that $\lambda_{1,p}^{\min}({\cal Q}^a, \Omega)$ always admits at least one anisotropic extremizer. The last two theorems of this section deal with multiplicity of such extremizers. Both point to a close relationship between the numbers of anisotropic extremizers and the symmetries of the region $\Omega$. In particular, they answer (B1) for the case of disks and a family of rectangles.

\begin{teor} \label{T3}
Let $a \in (0, 1)$ and assume $\Omega$ is a disk $D$ centered at the origin. Then,

\[
\lambda_{1,p}^{\min}({\cal Q}^a, D) = \lambda_{1,p}^{Q_a}(D) = a^{p/2} \lambda_{1,p}({\cal E}_a),
\]
where ${\cal E}_a$ denotes the elliptical region $\{(x, \sqrt{a}\, y):\ (x,y) \in D\}$.

Moreover, the set of all anisotropic extremizers is precisely given by $\{Q_\alpha:\ \alpha \in [a, 1]\}$.
\end{teor}

\begin{teor} \label{T4}
Let $a \in (0, 1)$ and assume $\Omega$ is the rectangle $R_a = [-1,1] \times [-\frac{1}{\sqrt{a}}, \frac{1}{\sqrt{a}}]$. Then,

\[
\lambda_{1,p}^{\min}({\cal Q}^a, R_a) = \lambda_{1,p}^{Q_a}(R_a) = a^{p/2} \lambda_{1,p}(R).
\]
where $R$ denotes the square $[-1,1] \times [-1, 1]$.

Moreover, the unique anisotropic extremizers are $Q_a$ and $Q_1$.
\end{teor}

\section{Optimization problems associated to $\lambda_{1,p}^{\max}({\cal Q}^a, \Omega)$}

We begin with a monotonicity result which will be used in the proof of Theorems \ref{T1}, \ref{T2} and \ref{T3a} and Corollary \ref{C.1}.

\begin{propo} \label{P0}
Let $Q_1, Q_2 \in {\cal Q}$. It holds that $\lambda_{1,p}^{Q_1}(\Omega) \leq \lambda_{1,p}^{Q_2}(\Omega)$ whenever $Q_1 \leq Q_2$ in $\R^2$.

Moreover, $\lambda_{1,p}^{Q_1}(\Omega) = \lambda_{1,p}^{Q_2}(\Omega)$ only when $Q_1 = Q_2$ in $\R^2$ provided that $\partial \Omega$ is $C^{1,\alpha}$ by parts.
\end{propo}

For the proof of the second part of this statement we recall a simple property about nonnegative quadratic forms.

\begin{lema} \label{L}
Let $Q$ be a nonnegative quadratic form on $\R^2$. If $Q \not\equiv 0$, then either $Q > 0$ in $\R^2 \setminus \{(0,0)\}$ or the kernel of $Q$ is a straight line in $\R^2$.
\end{lema}

\begin{proof}
Let $M$ be a symmetric matrix that represents $Q$, that is, $Q(v) = \langle v, M v \rangle$ with $v = (x,y)$. Using that $Q \geq 0$ we ensure that $Q(v_0) = 0$ if, and only if, $M v_0 = 0$. Assume $Q(v_0) = 0$ and consider the function $h_v(t) = Q(tv + v_0)$ for $t \in \R$, where $v$ is any nonzero fixed vector in $\R^2$. Noting that

\[
h_v(t) = t^2 \langle v, M v \rangle + 2 t \langle v, M v_0 \rangle + \langle v_0, M v_0 \rangle.
\]
and $t = 0$ is a minimum point of $h_v(t)$, we have $h_v^\prime(0) = 0$, or equivalently $\langle v, M v_0 \rangle = 0$. Since $v$ is an arbitrary vector, it follows that $M v_0 = 0$. The reciprocal is immediate.

Assume now that $Q \not\equiv 0$. For the conclusion, it suffices to show that there are no two linearly independent vectors in the kernel of $Q$. Let $v_1, v_2$ be nonzero vectors such that $Q(v_1) = 0 = Q(v_2)$. Then, from the above claim, we have $M v_1 = 0 = M v_2$. If $v_1$ and $v_2$ are linearly independent, then they span $\R^2$ and therefore $M \equiv 0$, so $Q \equiv 0$, contradicting the hypothesis of the statement.
\end{proof}

\begin{proof}[Proof of Proposition \ref{P0}]
The claim $\lambda_{1,p}^{Q_1}(\Omega) \leq \lambda_{1,p}^{Q_2}(\Omega)$ is immediate once it follows readily of the variational characterization. The part that needs to be carefully checked is that equality $\lambda_{1,p}^{Q_1}(\Omega) = \lambda_{1,p}^{Q_2}(\Omega)$ implies $Q_1 = Q_2$ in $\R^2$. Proceeding by contradiction and applying Lemma \ref{L} to the nonnegative quadratic form $Q = Q_2 - Q_1$, it follows that $Q_1 < Q_2$ at least in $\R^2$ except a straight line passing through the origin.

Let $\varphi_p \in W^{1,p}_0(\Omega) \cap C^{1,\beta}(\Omega)$ be a positive minimizer for $\lambda^{Q_1}_{1,p}(\Omega)$ normalized by $\Vert \varphi_p \Vert_p = 1$. Using the assumption that $\partial \Omega$ is $C^{1,\alpha}$ by parts, we recall that $\varphi_p$ extends $C^{1, \beta}$ until the boundary up to a finite number of points. Then, applying the Hopf's Lemma to the quasilinear elliptic operator $\Delta^{Q_1}_p$, we guarantee that $\nabla \varphi_p$ is a nonzero vector on each $C^{1,\alpha}$ part of $\partial \Omega$ pointing inward of $\Omega$. Since $\Omega$ is a domain, we can choose a point $(x_0,y_0)$ on some smooth part of $\partial \Omega$ so that $Q_1(\nabla \varphi_p(x_0,y_0)) < Q_2(\nabla \varphi_p(x_0,y_0))$. It is clear that the strict inequality $Q_1(\nabla \varphi_p) < Q_2(\nabla \varphi_p)$ remains in the open $\Omega_\delta = \{(x,y) \in \Omega:\ \vert (x - x_0, y - y_0) \vert < \delta\}$ for $\delta > 0$ small enough. Using the strict inequality in $\Omega_\delta$ and $Q_1(\nabla \varphi_p) \leq Q_2(\nabla \varphi_p)$ in $\Omega \setminus \Omega_\delta$, we get

\begin{eqnarray*}
\lambda^{Q_1}_{1,p}(\Omega) &\leq& \iint_\Omega Q_1^{\frac p2}(\nabla \varphi_p)\, dA = \iint_{\Omega_\delta} Q_1^{\frac p2}(\nabla \varphi_p)\, dA + \iint_{\Omega \setminus \Omega_\delta} Q_1^{\frac p2}(\nabla \varphi_p)\, dA \\
&\leq& \iint_{\Omega_\delta} Q_1^{\frac p2}(\nabla \varphi_p)\, dA + \iint_{\Omega \setminus \Omega_\delta} Q_2^{\frac p2}(\nabla \varphi_p)\, dA \\
&<& \iint_{\Omega_\delta} Q_2^{\frac p2}(\nabla \varphi_p)\, dA + \iint_{\Omega \setminus \Omega_\delta} Q_2^{\frac p2}(\nabla \varphi_p)\, dA = \lambda_{1,p}^{Q_2}(\Omega).
\end{eqnarray*}
\end{proof}

\begin{proof}[Proof of Theorem \ref{T1}]
By definition, for any $Q \in {\cal Q}^a$, we have $Q(x,y) \leq \overline{Q}(x,y) := x^2 + y^2$ for all $(x,y) \in \R^2$ and so, by Proposition \ref{P0},

\[
\lambda_{1,p}^{\max}({\cal Q}^a, \Omega) = \lambda_{1,p}^{\overline{Q}}(\Omega) = \lambda_{1,p}(\Omega)
\]
since $\overline{Q} \in {\cal Q}^a$. Moreover, by the same proposition, the strict inequality $\lambda_{1,p}^Q(\Omega) <  \lambda_{1,p}^{\overline{Q}}(\Omega)$ holds whenever $Q \not \equiv \overline{Q}$, so that $\overline{Q}$ is the unique anisotropic extremizer for $\lambda_{1,p}^{\max}({\cal Q}^a, \Omega)$.
\end{proof}

\section{Optimization problems associated to $\lambda_{1,p}^{\min}({\cal Q}^a, \Omega)$}

First we develop all the necessary ingredients to prove the existence and shape of anisotropic extremizers for $\lambda_{1,p}^{\min}({\cal Q}^a, \Omega)$ on ${\cal Q}^a$.

\subsection{The structure of the set ${\cal Q}^a$}
Let

\[
{\cal Q}_a := \{Q(x,y) \in {\cal Q}^a:\, Q_{\min} = a\}.
\]
Clearly, we have ${\cal Q}_a \subset {\cal Q}^a$. The next result describes precisely the set ${\cal Q}_a$.

\begin{propo} \label{P1}
For each $a \in (0,1)$, the set ${\cal Q}_a$ is precisely given by $\{Q_\alpha:\ \alpha \in [a, 1]\}$, where $Q_\alpha$ is as defined in Section 3.
\end{propo}

\begin{proof}
Let $Q = \alpha x^2 + 2 \beta xy + \gamma y^2$ with $\alpha, \gamma > 0$, $\beta \geq 0$ and $\beta^2 < \alpha \gamma$. Consider an orthogonal matrix $A$ diagonalizing the symmetric matrix

\[
M = \begin{bmatrix}
\alpha & \beta\\
\beta & \gamma
\end{bmatrix},
\]
into its real eigenvalues $\mu_1$ and $\mu_2$ which can be assumed $\mu_1 \leq \mu_2$. One easily checks that $(Q \circ A)(x,y) = \mu_1 x^2 + \mu_2 y^2$. Hence, $Q \in {\cal Q}_a$ if, and only if, $\mu_1 = a$ and $\mu_2 = 1$. But this is equivalent to say that $\beta$ and $\gamma$ satisfy the system

\begin{equation*}
\begin{cases}
(\alpha - a) (\gamma - a) = \beta^2& \\
(\alpha - 1) (\gamma - 1) = \beta^2,&
\end{cases}
\end{equation*}
whose solution is given for each $\alpha \in [a, 1]$ by

\[
\beta(\alpha) = \sqrt{(1 - \alpha) (\alpha - a)}\ \text{ and }\ \gamma(\alpha) = 1 + a - \alpha.
\]
\end{proof}

A direct consequence of this result is

\begin{cor} \label{C.2}
Let $a \in (0,1)$. For any $\alpha \in [a, 1]$, we have $Q_\alpha \circ A_\alpha = Q_a$, where $A_\alpha$ denotes the orthogonal matrix

\[
A_\alpha = \begin{bmatrix}
\sqrt{\frac{1 - \alpha}{1 - a}} & \sqrt{\frac{\alpha - a}{1 - a}}\\
-\sqrt{\frac{\alpha - a}{1 - a}} & \sqrt{\frac{1 - \alpha}{1 - a}}
\end{bmatrix}.
\]
In particular, there is a natural bijection between ${\cal Q}_a$ and $\{ A_\alpha:\ \alpha \in [a, 1]\}$ which in turn is in correspondence with the set of rotations ${\cal O}$ defined in the introduction.
\end{cor}

\begin{proof}
It is immediate to see that the columns of $A_\alpha$ are eigenvectors de $M$ associated to $a$ and $1$ in this order and also that the set $\{ A_\alpha:\ \alpha \in [a, 1]\}$ coincides with ${\cal O} = \{ R_\theta:\ \theta \in [0, \frac{\pi}{2}]\}$, where

\[
R_\theta = \begin{bmatrix}
\cos \theta & \sin \theta\\
-\sin \theta & \cos \theta
\end{bmatrix}.
\]
\end{proof}

A convenient way to look at the set ${\cal Q}^a$ is through the disjoint decomposition $\dot\bigcup_{b \in [a,1]} {\cal Q}_b$. Using this viewpoint, we establish the following key relation:

\begin{teor} \label{T6}
Let $a \in (0,1)$. Given any $Q \in {\cal Q}^a$ there is a function $Q_\alpha \in {\cal Q}_a$ such that

\[
Q(x,y) = \frac{1 - b}{1 - a} Q_\alpha(x,y) + \frac{b - a}{1 - a} (x^2 + y^2),\ \forall (x,y) \in \R^2,
\]
where $b \in [a,1]$ is the number for which $Q \in {\cal Q}_b$.
\end{teor}
Particularly, this statement implies that $Q \geq Q_\alpha$ in $\R^2$ and $Q \equiv Q_\alpha$ if, and only if, $b = a$.

\begin{proof}
Let $Q \in {\cal Q}^a$ and let $b \in [a,1]$ be such that $Q \in {\cal Q}_b$. Writing $Q(x,y) = \bar{\alpha} x^2 + 2 \bar{\beta} xy + \bar{\gamma} y^2$ with $\bar{\alpha}, \bar{\gamma} > 0$, $\bar{\beta} \geq 0$ and $\bar{\beta}^2 < \bar{\alpha} \bar{\gamma}$, by Proposition \ref{P1}, we know that

\[
\bar{\alpha} \in [b, 1],\ \ \bar{\beta} = \sqrt{(1 - \bar{\alpha}) (\bar{\alpha} - b)}\ \text{ and }\ \bar{\gamma} = 1 + b - \bar{\alpha}.
\]
Let us seek a function $Q_\alpha \in {\cal Q}_a$ such that

\begin{equation}\label{1}
Q \geq Q_\alpha \ \ {\rm in}\ \R^2.
\end{equation}
Again by that proposition, we have $Q_\alpha(x,y) = \alpha x^2 + 2 \beta xy + \gamma y^2$ where

\[
\alpha \in [a, 1],\ \ \beta = \sqrt{(1 - \alpha) (\alpha - a)}\ \text{ and }\ \gamma = 1 + a - \alpha.
\]
Notice that the inequality \eqref{1} holds if, and only if,

\begin{equation}\label{2}
(\bar{\beta} - \beta)^2 - (\bar{\alpha} - \alpha) (\bar{\gamma} - \gamma) \leq 0.
\end{equation}
The idea now is to prove the equivalence between this inequality and the relation

\begin{equation}\label{3}
\alpha = \frac{1 - a}{1 - b} \bar{\alpha} + \frac{a - b}{1 - b}.
\end{equation}
Replacing $\bar{\beta}$ and $\bar{\gamma}$ in function of $\bar{\alpha}$ and $\beta$ and $\gamma$ in function of $\alpha$ in \eqref{2}, after some development, we arrive at

\[
(1 - \alpha) (\bar{\alpha} - b) + (1 - \bar{\alpha}) (\alpha - a) - 2 \sqrt{(1 - \alpha) (\bar{\alpha} - b) (1 - \bar{\alpha}) (\alpha - a)} \leq 0.
\]
On the other hand, this inequality can be rewritten as

\[
(\sqrt{(1 - \alpha) (\bar{\alpha} - b)} - \sqrt{(1 - \bar{\alpha}) (\alpha - a)})^2 \leq 0.
\]
Consequently,

\[
(1 - \alpha) (\bar{\alpha} - b) = (1 - \bar{\alpha}) (\alpha - a)
\]
which yields the value of $\alpha$ as in \eqref{3}. This equation also allows us to express $\bar{\alpha}$, $\bar{\beta}$ and $\bar{\gamma}$ in terms respectively of $\alpha$, $\beta$ and $\gamma$. Indeed, it can be placed into three suitable ways:

\begin{equation}\label{4}
\bar{\alpha} = \frac{1 - b}{1 - a} \alpha + \frac{b - a}{1 - a},
\end{equation}

\[
1 - \bar{\alpha} = \frac{1 - b}{1 - a} (1 - \alpha)\ \ {\rm and}\ \ \bar{\alpha} - b = \frac{1 - b}{1 - a} (\alpha - a).
\]
The last equalities clearly imply that $\alpha \in [a,1]$ if, and only if, $\beta \in [b,1]$. Moreover, they produce the relations

\begin{equation}\label{5}
\bar{\beta} = \sqrt{(1 - \bar{\alpha}) (\bar{\alpha} - b)} = \frac{1 - b}{1 - a} \sqrt{(1 - \alpha) (\alpha - a)} = \frac{1 - b}{1 - a} \beta
\end{equation}
and

\begin{equation}\label{6}
\bar{\gamma} = 1 + b - \bar{\alpha} = \frac{1 - b}{1 - a} (1 - \alpha) + b = \frac{1 - b}{1 - a} (1 + a - \alpha) + \frac{b - a}{1 - a} = \frac{1 - b}{1 - a} \gamma + \frac{b - a}{1 - a}.
\end{equation}
Plugging \eqref{4}, \eqref{5} and \eqref{6} in $Q$, we deduce that

\[
Q(x,y) = \frac{1 - b}{1 - a} Q_\alpha(x,y) + \frac{b - a}{1 - a} (x^2 + y^2).
\]
\end{proof}

\subsection{Proof of Theorem \ref{T2} and Corollary \ref{C.1}}

We first prove Theorem \ref{T2} and after we invoke it in the proof of Corollary \ref{C.1}.

\begin{proof}[Proof of Theorem \ref{T2}]

Let $Q \in {\cal Q}^a$ and assume that $Q \in {\cal Q}_b$ for some $b \in (a, 1]$. By Theorem \ref{T6}, there exists $Q_\alpha \in {\cal Q}_a$ such that $Q \geq Q_\alpha$ and $Q \not\equiv Q_\alpha$ since $b > a$. By Proposition \ref{P0}, we have in general that $\lambda_{1,p}^Q(\Omega) \geq \lambda_{1,p}^{Q_\alpha}(\Omega)$. Therefore, it follows that

\[
\lambda_{1,p}^{\min}({\cal Q}^a,\Omega) = \inf_{Q \in {\cal Q}_a} \lambda_{1,p}^Q(\Omega)
\]
and moreover, by assuming $\partial \Omega$ is $C^{1,\alpha}$ by parts, the above monotonicity is strict, so that any anisotropic extremizer, if exist, belongs to ${\cal Q}_a$.

On the other hand, by Corollary \ref{C.2}, each function $Q \in {\cal Q}_a$ corresponds to an orthogonal matrix $A \in {\cal O}$ such that $Q \circ A = Q_a$. Then, using the change of variable $v(X^\prime) = u(X)$ for $X^\prime = A^T X \in \Omega_A$ with $X = (x,y)$ and $X^\prime = (x^\prime,y^\prime)$, we obtain

\[
\iint_\Omega Q^{\frac p2}(\nabla u)\, dA = \iint_{\Omega_A} Q_a^{\frac p2}(\nabla v)\, dA\ \text{ and }\ \iint_\Omega \vert u \vert^p\, dA = \iint_{\Omega_A} \vert v \vert^p\, dA,
\]
which yields $\lambda_{1,p}^Q(\Omega) = \lambda_{1,p}^{Q_a}(\Omega_A)$. Making now the change $v(x^\prime,y^\prime) = u(x,y)$ for $x^\prime = x$ and $y^\prime = \sqrt{a}y$ for $(x^\prime,y^\prime) \in \Omega^a_A$, we get

\[
\iint_{\Omega_A} Q_a^{\frac p2}(\nabla u)\, dA = a^{(p - 1)/2} \iint_{\Omega^a_A} \vert \nabla v \vert^p\, dA\ \text{ and }\ \iint_{\Omega_A} \vert u \vert^p\, dA = a^{-1/2}\iint_{\Omega^a_A} \vert v \vert^p\, dA,
\]
and so $\lambda_{1,p}^{Q_a}(\Omega_A) = a^{p/2} \lambda_{1,p}(\Omega^a_A)$. Consequently, we derive

\[
\lambda_{1,p}^{\min}({\cal Q}^a,\Omega) = \inf_{A \in {\cal O}} \lambda_{1,p}^{Q_a}(\Omega_A) = a^{p/2} \inf_{A \in {\cal O}} \lambda_{1,p}(\Omega^a_A).
\]
Notice now that ${\cal O}$ is a closed subset of the compact set of all orthogonal $2 \times 2$ matrices and the first eigenvalue of the $p$-Laplace operator $\lambda_{1,p}(\Omega)$ depends continuously on $\Omega$ with respect to the Hausdorff topology (see Theorem 3.2 in \cite{L}). Hence, the above infimum is always attained and therefore,

\[
\lambda_{1,p}^{\min}({\cal Q}^a,\Omega) = \min_{A \in {\cal O}} \lambda_{1,p}^{Q_a}(\Omega_A) = a^{p/2} \min_{A \in {\cal O}} \lambda_{1,p}(\Omega^a_A).
\]
Moreover, each minimizer $A \in {\cal O}$ is associated to the anisotropic extremizer $Q = Q_a \circ A^T \in {\cal Q}_a$ for $\lambda_{1,p}^{\min}({\cal Q}^a,\Omega)$.
\end{proof}

\begin{proof}[Proof of Corollary \ref{C.1}]

Let $a \in (0, 1)$. Consider the functions $\underline{Q}(x,y) = a(x^2 + y^2)$ and $\overline{Q}(x,y) = x^2 + y^2$. By Theorems \ref{T1} and \ref{T2}, we know that

\[
\lambda_{1,p}^{\min}({\cal Q}^a,\Omega) = \lambda_{1,p}^{Q_\alpha}(\Omega)\ \text{ and } \ \lambda_{1,p}^{\max}({\cal Q}^a,\Omega) = \lambda_{1,p}^{\overline{Q}}(\Omega)
\]
for some $Q_\alpha \in {\cal Q}_a$. From the definition of ${\cal Q}_a$, one knows that $\underline{Q} \lneqq Q_\alpha \lneqq \overline{Q}$ in $\R^2$. Since $\partial \Omega$ is $C^{1,\alpha}$ by parts, by Proposition \ref{P0}, we get

\[
a^{p/2} \lambda_{1,p}(\Omega) = \lambda_{1,p}^{\underline{Q}}(\Omega) < \lambda_{1,p}^{Q_\alpha}(\Omega) <  \lambda_{1,p}^{\overline{Q}}(\Omega),
\]
so that

\[
a^{p/2} \lambda_{1,p}(\Omega) < \lambda_{1,p}^{\min}({\cal Q}^a,\Omega) < \lambda_{1,p}^{\max}({\cal Q}^a, \Omega).
\]
\end{proof}

\subsection{Proof of Theorem \ref{T3a}}
We start proving the positivity of $\lambda_{1,p}^{\min}({\cal Q}^0,\Omega)$. For any $a \in (0,1]$, due to Theorem~\ref{T2}, we have $\lambda_{1,p}^{\min}({\cal Q}^a,\Omega) = \lambda_{1,p}^{Q_{a}}(\Omega_A)$ for some rotation $A \in {\cal O}$. If $u_a \in W_0^{1,p}(\Omega)$ with $\lVert u_a \rVert_p = 1$ is an eigenfunction associated to $\lambda_{1,p}^{Q_{a}}(\Omega_A)$, then

\begin{align*}
\lambda_{1,p}^{\min}({\cal Q}^a,\Omega)  = \lambda_{1,p}^{Q_{a}}(\Omega_A) &=  \iint_{\Omega_A} Q_a^{\frac p2}(\nabla u_{a})\, dA \\
&= \iint_{\Omega_A} \left(a \vert D_x u_a \vert^2 + \vert D_y u_a \vert^2 \right)^{\frac{p}{2}}\, dA \\
&\geq \iint_{\Omega_A} \vert D_y u_a \vert^p\, dA \\
&\geq d(p,\Omega_A) \iint_{\Omega_A} |u_a|^p \,dA = d(p,\Omega_A),
\end{align*}
where $d(p,\Omega_A) = t_p^p w(\Omega_A, \xi)^{-p}$ is the positive constant given in Lemma 1 of \cite{HJM} with $\xi = (0,1)$.  Hence, using that ${\cal O}$ is compact and $w(\Omega_A, \xi)$ depends continuously on $A$, one has

\[
d_0(p,\Omega) := \inf_{A \in {\cal O}} d(p,\Omega_A) = \min_{A \in {\cal O}}  d(p,\Omega_A) > 0,
\]
so that

\[
\lambda_{1,p}^{\min}({\cal Q}^0,\Omega) = \inf_{a \in (0,1]} \lambda_{1,p}^{\min}({\cal Q}^a,\Omega) \geq d_0(p,\Omega) > 0.
\]

The second statement is proved by contradiction. Assume that there is $Q \in {\cal Q}^0$ such that $\lambda_{1,p}^{\min}({\cal Q}^0,\Omega) = \lambda_{1,p}^{Q}(\Omega)$. Set $a = Q_{\min} > 0$, so $\lambda_{1,p}^{\min}({\cal Q}^0,\Omega) = \lambda_{1,p}^{\min}({\cal Q}^a,\Omega)$. From Theorem~\ref{T2}, we know that $\lambda_{1,p}^{\min}({\cal Q}^a,\Omega) = \lambda_{1,p}^{Q_a}(\Omega_A)$ for some $A \in {\cal O}$. Since $Q_a \geq Q_{a/2}$ in $\R^2$ and $Q_a \not\equiv Q_{a/2}$, by Proposition~\ref{P0}, we get

\[
\lambda_{1,p}^{\min}({\cal Q}^0,\Omega) = \lambda_{1,p}^{Q_a}(\Omega_A) > \lambda_{1,p}^{Q_{a/2}}(\Omega_A) = \lambda_{1,p}^{\tilde{Q}}(\Omega)
\]
with $\tilde{Q} = Q_{a/2} \circ A^T \in {\cal Q}^0$, which is a contradiction.

\subsection{Proof of Theorems \ref{T20} and \ref{T201}}

\begin{proof}[Proof of Theorem \ref{T20}]
Let $a \in (0,1)$ and $b \in [a,1)$. Given $Q \in {\cal Q}_b$, by Theorem~\ref{T6}, there is a function $Q_{\alpha} \in {\cal Q}_a$ such that

\[
Q(x,y) = \frac{1-b}{1-a}Q_{\alpha}(x,y) + \frac{b-a}{1-a}|(x,y)|^2.
\]
Let $\varphi_p \in W_0^{1,p}(\Omega)$ with $\lVert \varphi_p \rVert_p = 1$ such that

\[
\iint_{\Omega} Q_{\alpha}^{\frac{p}{2}}(\nabla \varphi_p)\, dA = \lambda_{1,p}^{\min}({\cal Q}^{a},\Omega).
\]
Then using the mean value theorem,

\begin{align*}
\lambda_{1,p}^{\min}({\cal Q}^{b},\Omega) - \lambda_{1,p}^{\min}({\cal Q}^{a},\Omega) &\leq \iint_{\Omega}Q^{\frac{p}{2}}(\nabla \varphi_p)\, dA - \iint_{\Omega}Q_{\alpha}^{\frac{p}{2}}(\nabla \varphi_p)\, dA = \iint_{\Omega}Q^{\frac{p}{2}}(\nabla \varphi_p) - Q_{\alpha}^{\frac{p}{2}}(\nabla \varphi_p)\, dA\\
&\leq \iint_{\Omega}pQ^{\frac{p-1}{2}}(\nabla \varphi_p)\left(Q^{\frac{1}{2}}(\nabla \varphi_p) - Q_{\alpha}^{\frac{1}{2}}(\nabla \varphi_p)\right)\, dA \\
&\leq p\iint_{\Omega}Q^{\frac{p-1}{2}}(\nabla \varphi_p)\left(\sqrt{\frac{1-b}{1-a}}Q_{\alpha}^{\frac{1}{2}}(\nabla \varphi_p) + \sqrt{\frac{b-a}{1-a}}|\nabla \varphi_p| - Q_{\alpha}^{\frac{1}{2}}(\nabla \varphi_p)\right)\, dA \\
&\leq \frac{p}{\sqrt{1-a}}\iint_{\Omega}Q^{\frac{p-1}{2}}(\nabla \varphi_p)\left(\left(\sqrt{1-b}-\sqrt{1-a}\right)Q_{\alpha}^{\frac{1}{2}}(\nabla \varphi_p)+\sqrt{b-a}|\nabla \varphi_p|\right)\, dA\\
&\leq \frac{p}{\sqrt{1-a}}\iint_{\Omega}Q^{\frac{p-1}{2}}(\nabla \varphi_p)\sqrt{b-a}|\nabla \varphi_p|\, dA \\
&\leq \frac{p\sqrt{b-a}}{\sqrt{1-a}}\iint_{\Omega}Q^{\frac{p-1}{2}}(\nabla \varphi_p)\frac{Q_{\alpha}^{\frac{1}{2}}(\nabla \varphi_p)}{\sqrt{a}}\, dA\\
&= \frac{p\sqrt{b-a}}{\sqrt{a(1-a)}}\iint_{\Omega} \left(\frac{1-b}{1-a}Q_{\alpha}(\nabla \varphi_p) + \frac{b-a}{1-a}|\nabla \varphi_p|^2\right)^{\frac{p-1}{2}}Q_{\alpha}^{\frac{1}{2}}(\nabla \varphi_p)\, dA \\
&\leq \frac{p\sqrt{b-a}}{\sqrt{a(1-a)^{p}}}\iint_{\Omega}\left(1-b+\frac{b-a}{a}\right)^{\frac{p-1}{2}}Q_{\alpha}^{\frac{p}{2}}(\nabla \varphi_p)\, dA\\
&= p\sqrt{\frac{b^{p-1}(b-a)}{a^p(1-a)}}\ \lambda_{1,p}^{\min}({\cal Q}^{a},\Omega).
\end{align*}
Hence, we derive the upper quantitative inequality

\[
\frac{\lambda_{1,p}^{\min}({\cal Q}^{b},\Omega)}{\lambda_{1,p}^{\min}({\cal Q}^{a},\Omega)} - 1 \leq p\sqrt{\frac{b^{p-1}(b-a)}{a^p(1-a)}},
\]
which is sharp since equality in the above third inequality only holds when $b = a$ since $b < 1$.
\end{proof}

\begin{proof}[Proof of Theorem \ref{T201}]
We first get a lower quantitative inequality for $p \geq 2$. Let $Q \in {\cal Q}_{b}$ and choose a positive function $\psi_p \in W_{0}^{1,p}(\Omega)$ with $\lVert \psi_p \rVert_{p} = 1$ such that

\[
\lambda_{1,p}^{\min}({\cal Q}^b,\Omega) = \iint_{\Omega}Q^{\frac{p}{2}}(\nabla \psi_p)\, dA.
\]
Again taking $Q_{\alpha} \in {\cal Q}_a$ as in Theorem~\ref{T6} and applying the mean value theorem, one gets

\begin{align*}
   \lambda_{1,p}^{\min}({\cal Q}^b,\Omega) - \lambda_{1,p}^{\min}({\cal Q}^a,\Omega) &\geq \iint_{\Omega}Q^{\frac{p}{2}}(\nabla \psi_p)\, dA - \iint_{\Omega}Q_{\alpha}^{\frac{p}{2}}(\nabla \psi_p)\, dA \\
   &\geq \iint_{\Omega}\frac{p}{2}Q_{\alpha}^{\frac{p}{2}-1}(\nabla \psi_p)\left(Q(\nabla \psi_p) - Q_{\alpha}(\nabla \psi_p)\right)\, dA \\
   &= \frac{p}{2}\iint_{\Omega}Q_{\alpha}^{\frac{p}{2}-1}(\nabla \psi_p)\left(\frac{1-b}{1-a}Q_{\alpha}(\nabla \psi_p) + \frac{b-a}{1-a}|\nabla \psi_p|^2 -Q_{\alpha}(\nabla \psi_p) \right)\, dA\\
   &= \frac{p}{2}\iint_{\Omega}Q_{\alpha}^{\frac{p}{2}-1}(\nabla \psi_p)\frac{b-a}{1-a}\left(|\nabla \psi_p|^2 - Q_{\alpha}(\nabla \psi_p)\right)\, dA\\
   &\geq \frac{pa^{(p-2)/2}(b-a)}{2(1-a)}\iint_{\Omega}|\nabla \psi_p|^{p-2}\left(|\nabla \psi_p|^2 - Q_{\alpha}(\nabla \psi_p)\right)\, dA\\
   &= \frac{pa^{(p-2)/2}(b-a)}{2(1-a)}\iint_{\Omega_A}|\nabla v_p|^{p-2}\left(|\nabla v_p|^2 - a \vert D_x v_p \vert^2 - \vert D_y v_p \vert^2)\right)\, dA \\
   &= \frac{pa^{(p-2)/2}(b-a)}{2} \iint_{\Omega_A}|\nabla v_p|^{p-2} \vert D_x v_p \vert^2\, dA\\
   &\geq \frac{pa^{(p-2)/2}(b-a)}{2} \iint_{\Omega_A} \vert D_x v_p \vert^p\, dA,
\end{align*}
where $A \in {\cal O}$ is the matrix so that $Q_\alpha \circ A = Q_a$ and $v_p = \psi_p \circ A$. Therefore,
\[
\lambda_{1,p}^{\min}({\cal Q}^b,\Omega) - \lambda_{1,p}^{\min}({\cal Q}^a,\Omega) \geq \frac{pa^{(p-2)/2}(b-a)}{2} c_0(p,\Omega),
\]
where, as in the statement,

\[
c_0(p,\Omega):= \inf_{A \in {\cal O}}\ \inf_{u \in W^{1,p}_0(\Omega)} \ \left\{ \iint_{\Omega_A} \vert D_x u \vert^p\, dA:\ \Vert u \Vert_p = 1 \right\}.
\]

The case $1 < p < 2$ is treated in a similar manner. With the same notation, applying the mean value theorem and then reverse Hölder inequality, we have
\begin{align*}
   \lambda_{1,p}^{\min}({\cal Q}^b,\Omega) - \lambda_{1,p}^{\min}({\cal Q}^a,\Omega) &\geq \iint_{\Omega}Q^{\frac{p}{2}}(\nabla \psi_p) \ dA - \iint_{\Omega}Q_{\alpha}^{\frac{p}{2}}(\nabla u) \ dA \\
   &\geq \iint_{\Omega}\frac{p}{2}Q^{\frac{p}{2}-1}(\nabla \psi_p)\left(Q(\nabla \psi_p) - Q_{\alpha}(\nabla \psi_p)\right) \ dA\\
   &= \frac{p}{2}\iint_{\Omega}Q^{\frac{p}{2}-1}(\nabla \psi_p)\left(\frac{1-b}{1-a}Q_{\alpha}(\nabla \psi_p) + \frac{b-a}{1-a}|\nabla \psi_p|^2 -Q_{\alpha}(\nabla \psi_p) \right) \ dA \\
   &= \frac{p}{2}\iint_{\Omega}Q^{\frac{p}{2}-1}(\nabla \psi_p)\frac{b-a}{1-a}\left(|\nabla \psi_p|^2 - Q_{\alpha}(\nabla \psi_p)\right) \ dA\\
   &\geq \frac{p(b-a)}{2(1-a)}\iint_{\Omega}|\nabla \psi_p|^{p-2}\left(|\nabla \psi_p|^2 - Q_{\alpha}(\nabla \psi_p)\right) \ dA\\
   &= \frac{p(b-a)}{2(1-a)}\iint_{\Omega_A}|\nabla v_p|^{p-2}\left(|\nabla v_p|^2-a|D_{x}v_p|^2-|D_{y}v_p|^2\right) \ dA\\
   &= \frac{p(b-a)}{2}\iint_{\Omega_A}|\nabla v_p|^{p-2}|D_{x}v_p|^2 \ dA
\end{align*}

\begin{align*}
   &\geq \frac{p(b-a)}{2} \left(\iint_{\Omega_A}|\nabla v_p|^p \ dA\right)^{\frac{p-2}{p}}\left(\iint_{\Omega_A}|D_{x}v_p|^p \ dA\right)^{\frac{2}{p}} \\
   &= \frac{p(b-a)}{2} \left(\iint_{\Omega}|\nabla \psi_p|^p \ dA\right)^{\frac{p-2}{p}}\left(\iint_{\Omega_A}|D_{x}v_p|^p \ dA\right)^{\frac{2}{p}}\\
   &\geq \frac{p(b-a)}{2}\left(\iint_{\Omega}\frac{Q^{\frac{p}{2}}(\nabla \psi_p)}{b^{p/2}} \ dA\right)^{\frac{p-2}{p}}\left(\iint_{\Omega_A}|D_{x}v_p|^p \ dA\right)^{\frac{2}{p}}\\
   &= \frac{pb^{(2-p)/2}(b-a)}{2}\left(\lambda_{1,p}^{\min}({\cal Q}^b,\Omega)\right)^{\frac{p-2}{2}}\left(\iint_{\Omega_A}|D_{x}v_p|^p \ dA\right)^{\frac{2}{p}} \\
   &\geq \frac{pb^{(2-p)/2}(b-a)}{2} \lambda_{1,p}(\Omega)^{(p-2)/2} c_{0}(p,\Omega)^{2/p}.
\end{align*}
Besides, if $u \in W^{1,p}_{0}(\Omega)$ and $\lVert u\rVert_p = 1$, we have

\[
\iint_{\Omega_A} \vert D_x u \vert^p\, dA \geq c(p,\Omega_A) \iint_{\Omega_A} |u|^p \,dA = c(p,\Omega_A),
\]
where $c(p,\Omega_A) = t_p^p w(\Omega_A, \xi)^{-p}$ is the positive constant given in Lemma 1 of \cite{HJM} with $\xi = (1,0)$. Arguing exactly as in the proof of Theorem \ref{T3a}, we deduce that

\[
c_0(p,\Omega) \geq \inf_{A \in {\cal O}} c(p,\Omega_A) = \min_{A \in {\cal O}}  c(p,\Omega_A) > 0.
\]
Finally, if $\partial \Omega$ is of $C^{1,\alpha}$ class, then $\psi_p \in C^{1,\beta}(\overline{\Omega})$ and, thanks to the Hopf's Lemma, $\vert \nabla \psi_p \vert^{-1} \nabla \psi_p$ maps the unity circle $\s^1$. Consequently, equality in the second inequality of both cases implies that $Q(\nabla \psi_p) = Q_\alpha(\nabla \psi_p)$, so that $Q \equiv Q_\alpha$ and then $b = a$.
\end{proof}

\subsection{Proof of Theorems \ref{T3} and \ref{T4}}

\begin{proof}[Proof of Theorem \ref{T3}:]
When $\Omega$ is a disk $D$ centered at the origin, we have that $\Omega_A = D$ and $\Omega^a_A = {\cal E}_a$ for every $A \in {\cal O}$. Thus, $\lambda_{1,p}^{Q_a}(\Omega_A)$ doesn't depend on $A$ and is equal to the constant value $\lambda_{1,p}^{Q_a}(D) = a^{p/2} \lambda_{1,p}({\cal E}_a)$. Therefore, by Theorem \ref{T2}, it follows that
\[
\lambda_{1,p}^{\min}({\cal Q}^a,D) = \lambda_{1,p}^{Q_a}(D) = a^{p/2} \lambda_{1,p}({\cal E}_a)
\]
and the corresponding set of extremizers is

\[
{\cal Q}_a = \{Q_a \circ A^T:\, A \in {\cal O}\} = \{Q_\alpha:\, \alpha \in [a,1]\}.
\]
\end{proof}

The proof of Theorem \ref{T4} uses a result on minimization of $\lambda_{1,p}(R)$ over all quadrilaterals $R$ of same area. The Faber-Krahn inequality for quadrilaterals states that the minimum of $\lambda_{1,p}(R)$ is attained only at squares of same area. For $p = 2$, this is a classical theorem due to Pólya and  Szegö \cite{PS} whose proof is based on Steiner's symmetrization arguments. For $p \neq 2$, the result is also true and its proof is carried out step by step as in \cite{PS}, see for example the proof of Theorem 1.1 of \cite{OC} where the same method is employed in the nonlocal $p$-Laplace case.

\begin{proof}[Proof of Theorem \ref{T4}:]
Assume $\Omega = R_a = [-1,1] \times [-\frac{1}{\sqrt{a}}, \frac{1}{\sqrt{a}}]$ for $a \in (0,1)$. Notice that $\Omega^a_A$ is a quadrilateral of area $4$ for any $A \in {\cal O}$. Recall that the extremizers of $\lambda_{1,p}^{\min}({\cal Q}^a,R_a)$ are given by $Q_a \circ A^T$ where $A$ minimizes $\lambda_{1,p}(\Omega^a_A)$. On the other hand, there are only two matrices $A \in {\cal O}$ that transform $\Omega^a_A$ into a square, namely, the identity matrix $I$ and

\[
A = \begin{bmatrix}
0 & 1\\
-1 & 0
\end{bmatrix}.
\]
Therefore, the unique extremizers of $\lambda_{1,p}^{\min}({\cal Q}^a, R_a)$ are $Q_a$ and $Q_a \circ A^T = Q_1$. Finally, using the identity matrix, we derive

\[
\lambda_{1,p}^{\min}({\cal Q}^a, R_a) = \lambda_{1,p}^{Q_a}(R_a) = a^{p/2} \lambda_{1,p}(R),
\]
where $R = [-1,1] \times [-1,1]$.
\end{proof}

\vs{0.5cm}

\n {\bf Acknowledgments:} The first author was partially supported by CAPES (PROEX 88887.712161/2022-00) and the second author was partially supported by CNPq (PQ 307432/2023-8, Universal 404374/2023-9) and Fapemig (Universal APQ 01598-23).





\end{document}